\documentclass[a4paper,12pt]{article}
\usepackage{latexsym, amssymb, amsfonts,amsmath}
\usepackage{amssymb,latexsym}
\usepackage{amsfonts}
\usepackage{amsmath,amsthm,graphicx}
\newcommand{\be}{\begin{equation}}
\newcommand{\ee}{\end{equation}}
\newcommand{\bea}{\begin{eqnarray}}
\newcommand{\eea}{\end{eqnarray}}
\newcommand{\bean}{\begin{eqnarray*}}
\newcommand{\eean}{\end{eqnarray*}}
\newcommand{\brray}{\begin{array}}
\newcommand{\erray}{\end{array}}
\newcommand{\ben}{\begin{equation}{nonumber}}
\newcommand{\een}{\end{equation}{nonumber}}

\newtheorem{dfn}{Definition}[section]
\newtheorem{thm}[dfn]{Theorem}
\newtheorem{lmma}[dfn]{Lemma}

\newtheorem{ppsn}[dfn]{Proposition}
\newtheorem{crlre}[dfn]{Corollary}
\newtheorem{xmpl}[dfn]{Example}
\newtheorem{rmrk}[dfn]{Remark}
\newcommand{\bdfn}{\begin{dfn}}
\newcommand{\bthm}{\begin{thm}}

\numberwithin{equation}{section}

\newcommand{\blr}{\begin{list}{$($\roman{cnt1}$)$} {\usecounter{cnt1}
        \setlength{\topsep}{0pt} \setlength{\itemsep}{0pt}}}
\newcommand{\bla}{\begin{list}{$($\alph{cnt2}$)$} {\usecounter{cnt2}
       \setlength{\topsep}{0pt} \setlength{\itemsep}{0pt}}}
\newcommand{\bln}{\begin{list}{$($\arabic{cnt3}$)$} {\usecounter{cnt3}
                \setlength{\topsep}{0pt} \setlength{\itemsep}{0pt}}}
\newcommand{\el}{\end{list}}

\newcommand{\blmma}{\begin{lmma}}
\newcommand{\bppsn}{\begin{ppsn}}
\newcommand{\bcrlre}{\begin{crlre}}
\newcommand{\bxmpl}{\begin{xmpl}}
\newcommand{\brmrk}{\begin{rmrk}}
\newcommand{\edfn}{\end{dfn}}
\newcommand{\ethm}{\end{thm}}
\newcommand{\elmma}{\end{lmma}}

\newcommand{\eppsn}{\end{ppsn}}
\newcommand{\ecrlre}{\end{crlre}}
\newcommand{\exmpl}{\end{xmpl}}
\newcommand{\ermrk}{\end{rmrk}}

\def\a*{{\cal A}_{h,*}}
\def\B{{\cal B}(h)}
\def\B1{{\cal B}_1(h)}
\def\b{{\cal B}^{\rm s.a.}(h)}
\def\b1{{\cal B}^{\rm s.a.}_1(h)}

\newcommand{\suml}{\sum \limits}
\newcommand{\itt}{\int \limits}

\begin{document}

 \begin{center}
 \Large{\bf{Koplienko Trace Formula}}\\
 \vspace{0.15in}
{\large Arup Chattopadhyay {\footnote {J.N.Centre for Advanced Scientific Research, Bangalore-560064, INDIA; ~~arup@jncasr.ac.in}} }
{\large and ~~~Kalyan B. Sinha {\footnote { J.N.Centre for Advanced Scientific Research and Indian Institute of Science, Bangalore-560064, INDIA; kbs@jncasr.ac.in}}}\\
\end{center}
\vspace{0.15in}
\begin{abstract}
In (\cite{kopleinko}), Koplienko gave a trace formula for perturbations of self-adjoint operators by operators of Hilbert-Schmidt class $\mathcal{B}_2(\mathcal{H})$.
Recently Gesztesy, Pushnitski and Simon (\cite{gesztesykop}) gave an alternative proof of the trace formula when the operators involved are bounded. 
In this article, we give a still another proof and extend the formula for unbounded case by reducing the problem to a finite dimensional one as in 
the proof of Krein trace formula by Voiculescu (\cite{Voiculescu}), Sinha and Mohapatra (\cite{Sinhamohapatra}).
\end{abstract}
\section{Introduction.}\label{sec: intro}
\emph{ \textbf{Notations:}} Here, $\mathcal{H}$ will denote the separable Hilbert space we work in; $\mathcal{B}(\mathcal{H})$, $\mathcal{B}_1(\mathcal{H})$,
$\mathcal{B}_2(\mathcal{H})$ the set of bounded, trace class, Hilbert-Schmidt class operators in $\mathcal{H}$ respectively with  $\|.\|, \|.\|_1, \|.\|_2$ as 
the associated norms. Let $H$ and $H_0$ be a pair of self-adjoint operators in $\mathcal{H}$ with $\sigma(H)$, $\sigma(H_0)$ as 
their spectra and $E(\lambda)$, $E_{0}(\lambda)$ the spectral families; and let $Dom(A)$, $Tr A$ be the domain of the operator $A$ and the trace 
of a trace class operator $A$ respectively.

Furthermore, if we assume that  $V\equiv H-H_0 \in \mathcal{B}_1(\mathcal{H})$, then Krein (\cite{mgkrein}) proved that there exists a unique real-valued $L^1(\mathbb{R})$- function $\xi$ 
with support in the interval $[a,b]$ ( where $a=min\{\inf \sigma(H), \inf \sigma(H_0)\}$ and $b=max\{\sup \sigma(H), \sup \sigma(H_0)\}$ ) such that 
\begin{equation}\label{eq: kreinintro}
 Tr\left[\phi\left(H\right)-\phi\left(H_0\right)\right] = \int_{a}^{b} \phi^{\prime}(\lambda)\xi(\lambda)d\lambda,
\end{equation}
for a large class of functions $\phi$ . 
The function $\xi$ is known as Krein's spectral shift function and the relation \eqref{eq: kreinintro} is called
Krein's trace formula. The original proof of Krein uses analytic function theory. In 1985, Voiculescu approached the trace formula \eqref{eq: kreinintro} from a different
direction. If $H$ and $H_0$ are bounded, then Voiculescu (\cite{Voiculescu}) proved that
\begin{equation}
 Tr\left[p\left(H\right)-p\left(H_0\right)\right] = \lim_{n\longrightarrow \infty} Tr_n\left[p\left(H_n\right)-p\left(H_{0,n}\right)\right],
\end{equation}
by adapting the Weyl-von Neumann's theorem (where $p(.)$ is a polynomial and $H_n,H_{0,n}$ are finite dimensional approximations of $H$ and $H_0$ respectively and
$Tr_n$ is the associated finite dimensional trace). Then one constructs the spectral shift function in the finite dimensional case and finally the formula is extended to the $\infty$-dimensional case. Later Sinha
and Mohapatra (\cite{Sinhamohapatra}) extended Voiculescu's method to the unbounded self-adjoint and unitary cases (\cite{smunitary}).
 
One can think of \eqref{eq: kreinintro} as a \textquotedblleft Mean Value theorem under trace for self-adjoint operators \textquotedblright and then a   
natural question arises if one can have a mean-value theorem under trace up to the next order. Koplienko (\cite{kopleinko}) indeed provided such a formula.
Let $H$ and $H_0$ be two self-adjoint operators in $\mathcal{H}$ such that $H-H_0 \equiv V\in \mathcal{B}_2(\mathcal{H})$. In this case the difference $\phi(H) - \phi(H_0)$
is no longer of trace-class and one has to consider instead
\begin{equation*}
 \phi(H) - \phi(H_0) - D\phi(H_0) \bullet V
\end{equation*}
where $D\phi(H_0)$ denotes the Frechet derivative of $\phi$ at $H_0$ ( see \cite{Bhatia}) and find a trace formula for the above expression. 
Under the above hypothesis, Koplienko's formula asserts that there exists a unique function $\eta \in L^1(\mathbb{R})$ such that
\begin{equation}\label{eq: intkopeq}
 Tr\{\phi(H) - \phi(H_0) - D\phi(H_0) \bullet V\} = \int_{-\infty}^{\infty} \phi''(\lambda)\eta(\lambda)d\lambda
\end{equation}
for rational functions $\phi$ with poles off $\mathbb{R}$. In 2007,
Gesztesy, Pushnitski and Simon (\cite{gesztesykop}) gave an alternative proof of the formula \eqref{eq: intkopeq} for the bounded case and in 2009, Dykema and
Skripka (\cite{DykemmaSkripkka}, \cite{Skripkkaunbounded}) obtained the formula \eqref{eq: intkopeq} in the semi-finite von Neumann algebra setting.
\vspace{0.01in}

Here we revisit the proof of Koplienko's formula for bounded case and prove the unbounded self-adjoint case, we believe for the first time,  using the idea of finite dimensional approximation as in the
works of Voiculescu, Sinha and Mohapatra, referred earlier. The plan of the paper is to first prove in section 2, Koplienko formula when $ dim \mathcal{H}< \infty$; section 3 is devoted
to the reduction of the problem to finite dimensions and in section 4 we prove the trace formula for both cases, viz. when the pairs $(H_0, H)$ are bounded
or unbounded self-adjoint.

\section{Koplienko formula in finite dimension}\label{sec: second}
\begin{thm}\label{thm: finitekop}
Let $H$ and $H_0$ be two self-adjoint operators in a Hilbert space $\mathcal{H}$ such that $H - H_0 \equiv V$ and let ~$p(\lambda) = \lambda^r (r\geq 2)$.
\vspace{0.1in}

(i) Then ~~~~$Dp(H_0)\bullet V = \suml_{j=0}^{r-1} H_0^{r-j-1}VH_0^j$ ~~~~ and ~~~~~$\frac{d}{ds}(p(H_s)) = \suml_{j=0}^{r-1} H_s^{r-j-1}VH_s^j$, where 
$H_s = H_0 +sV ~~(0\leq s\leq 1).$
\vspace{0.1in}

(ii) If furthermore $ dim \mathcal{H}< \infty$, then there exists a unique non-negative $L^1(\mathbb{R})$-function $\eta$ such that
\begin{equation} 
 Tr\{p(H)-p(H_0)-Dp(H_0)\bullet V\} = \itt_a^b p''(\lambda)\eta(\lambda)d\lambda,
\end{equation}
for some $-\infty< a < b < \infty$, where $p(.)$ is any polynomial on $[a,b]$ with complex coefficients.
\vspace{0.1in}

Moreover, 
\be\label{eq: itaexp}
\eta(\lambda) = \itt_0^1 Tr\{V\left[E_0(\lambda) - E_s (\lambda)\right]\}ds, 
\ee
where $E_s(.)$ is the spectral family of the self-adjoint operator $H_s$ , and
\be 
\|\eta\|_1=\frac{1}{2}\|V\|_2^2.
\ee
\vspace{0.1in}

(iii) For $dim \mathcal{H}< \infty$,
\begin{equation}\label{eq: finiteexp}
 Tr\{e^{itH} - e^{itH_0} - D(e^{itH_0})\bullet V\} = (it)^2\int_a^b e^{it\lambda}\eta(\lambda)d\lambda,
\end{equation}
for some $-\infty< a < b < \infty$, $t\in \mathbb{R}$ and $\eta$ is given by \eqref{eq: itaexp}.

\end{thm}
\begin{proof}
\emph{(i)} For $p(\lambda) = \lambda^r$($r\geq 2$),
\begin{equation}\label{eq: totalexpress}  
p(H_0+V) - p(H_0) = \sum_{j=0}^{r-1} H_0^{r-j-1}VH_0^j + \sum_{j=0}^{r-2}\sum_{k=0}^{r-j-2} H^{r-j-k-2}VH_0^{k}VH_0^j,
\end{equation}
and hence $$\|p(H_0+V) - p(H_0) - \sum_{j=0}^{r-1} H_0^{r-j-1}VH_0^j\| \leq \sum_{j=0}^{r-2}\sum_{k=0}^{r-j-2} \|H\|^{r-j-k-2}\|V\|\|H_0\|^{k}\|V\|\|H_0\|^j,$$
proving that $Dp(H_0)\bullet V =  \suml_{j=0}^{r-1} H_0^{r-j-1}VH_0^j$. A similar calculation shows that 
$$\frac{H_{s+h}^r - H_s^r}{h} = \suml_{j=0}^{r-1} \left[H_0+(s+h)V\right]^{r-j-1}VH_s^j,$$ which converges in operator norm to
$\suml_{j=0}^{r-1} H_s^{r-j-1}VH_s^j$ as $h\rightarrow 0$. 

\emph{(ii)} By using the cyclicity of trace and noting that the trace now is a finite sum, we have that for $p(\lambda) = \lambda^r$ ($r\geq 2$),
\begin{equation*}
\begin{split}
Tr\{p(H) - p(H_0) - Dp(H_0)\bullet V\}\\
& \hspace{-4cm} = Tr\left(\int_0^1 \frac{d}{ds}\left(p(H_s)\right)ds\right) - Tr\left(\sum_{j=0}^{r-1} H_0^{r-j-1}VH_0^j\right) \\
& \hspace{-4cm} = \int_0^1 rTr\left(VH_s^{r-1}\right)ds - \int_0^1 r Tr\left(VH_0^{r-1}\right)ds \\
& \hspace{-4cm} = Tr\left[ rV\int_0^1 ds \int_a^b \lambda^{r-1} \{E_s(d\lambda) - E_0(d\lambda)\}\right].\\
\end{split}
\end{equation*}
It is easy to see that there exists $a,b\in \mathbb{R} ~(-\infty <a <b < +\infty)$ such that $supp E_s(.) \subseteq [a,b]$ for all $s\in [0,1]$. By integrating by-parts
and noting that $E_s(.) - E_0(.) = 0$ for $\lambda =a,b$ , we have that
\begin{equation*}
\begin{split}
& Tr\{p(H) - p(H_0) - Dp(H_0)\bullet V\} \\
& = Tr\left[ rV \int_0^1 ds \left( \lambda^{r-1}\{E_s(\lambda) - E_0(\lambda)\}\mid_a^b - \int_a^b (r-1)\lambda^{r-2}\{E_s(\lambda) - E_0(\lambda)\}\right)\right]\\
& = \int_a^b r(r-1)\lambda^{r-2} \left( \int_0^1 Tr\{V[E_0(\lambda) - E_s(\lambda)]ds \right) d\lambda\\
& =\int_a^b p^{''}(\lambda) \eta(\lambda) d\lambda, \quad \text{where we have set} \quad  \eta(\lambda) = \itt_0^1 Tr\{V\left[E_0(\lambda) - E_s (\lambda)\right]\}ds. \\
\end{split}
\end{equation*}
To prove the positivity of ~$\eta(\lambda)$~, we use the idea of double spectral integrals, introduced by Birman-Solomyak (\cite{birmansolomyak}, \cite{birsolomyak}).
For fixed $\lambda$, and $\epsilon > 0$ define a smoothly non-increasing function $\phi_{\epsilon , \lambda}$ such that 
\begin{equation*}
\phi_{\epsilon, \lambda}(\alpha) =
\begin{cases}
0, &\text{if $\alpha \geq \lambda + \epsilon .$}\\
1, ~&\text{if $a\leq \alpha \leq \lambda$.}
\end{cases}
\end{equation*}
Therefore 
\begin{equation*}
\hspace{-5cm} \phi_{\epsilon, \lambda}(H_0) - \phi_{\epsilon, \lambda}(H_s) 
= \int_a^b \int_a^b \left[\phi_{\epsilon, \lambda}(\alpha) - \phi_{\epsilon, \lambda}(\beta)\right]E_0(d\alpha)E_s(d\beta)
\end{equation*}

\begin{equation}\label{eq: doubspec}
= -s \int_a^b \int_a^b \frac{\phi_{\epsilon, \lambda}(\alpha) - \phi_{\epsilon, \lambda}(\beta)}{\alpha - \beta} E_0(d\alpha)VE_s(d\beta) 
= -s \int_{[a,b]\times [a,b]} \frac{\phi_{\epsilon, \lambda}(\alpha) - \phi_{\epsilon, \lambda}(\beta)}{\alpha - \beta} \mathcal{G}(d\alpha \times d\beta).V,
\end{equation}
where $\mathcal{G} (\Delta \times \delta) X = E_0(\Delta)XE_s(\delta)$~~(~$X\in \mathcal{B}_2(\mathcal{H})$ and $ \Delta \times \delta \subseteq \mathbb{R} \times \mathbb{R})$
extends to a $\mathcal{B}_2(\mathcal{H})$- valued spectral measure in $\mathbb{R}^2$ with total $\mathcal{B}_2(\mathcal{H})$-variation less than or equal to 1. Thus
\begin{equation}\label{eq: dubspec}
Tr\{V\left[\phi_{\epsilon, \lambda}(H_0) - \phi_{\epsilon, \lambda}(H_s)\right]\}
= -s \int_a^b \int_a^b \frac{\phi_{\epsilon, \lambda}(\alpha) - \phi_{\epsilon, \lambda}(\beta)}{\alpha - \beta} Tr\{VE_0(d\alpha)VE_s(d\beta)\}.
\end{equation}
Since by construction, $\phi_{\epsilon, \lambda}$ is a non-increasing function, the integrand in \eqref{eq: dubspec} is non-positive and hence 
$$Tr\{V\left[\phi_{\epsilon, \lambda}(H_0) - \phi_{\epsilon, \lambda}(H_s)\right]\} \geq 0 ~~~~~~~~~~\forall ~~\lambda, \epsilon>0.$$
Furthermore, $\phi_{\epsilon, \lambda}(H_0)$ and $\phi_{\epsilon, \lambda}(H_s)$ converges strongly to $E_0(\lambda)$ and $E_s(\lambda)$ respectively 
as $\epsilon \rightarrow 0$, (spectral family is right continuous in our definition) and hence $$Tr\{V\left[E_0(\lambda) - E_s(\lambda)\right]\} \geq 0 \quad \text{for} \quad 0\leq s\leq 1.$$
Therefore $\eta(\lambda) \geq 0 \quad \text{for all} \quad  \lambda \in[a,b].$ The last conclusion is a consequence of the fact that 
$$\hspace{-2cm} \| \eta\|_1 = \int_a^b \eta(\lambda)d\lambda = \frac{1}{2} \int_a^b p''(\lambda) \eta(\lambda) d\lambda \quad \text{(where} \quad p(\lambda) = \lambda^2)$$
$$\hspace{-3.4cm} = \frac{1}{2} Tr\{H^2 - H_0^2 - D(H_0^2)\bullet V\} = \frac{1}{2} \|V\|_2^2.$$

\emph{(iii)} It is easy to verify that $$D(e^{itH_0})\bullet V = it\int_0^1 e^{it\alpha  H_0} V e^{it(1-\alpha)  H_0} d\alpha$$ and a calculation identical
to the one in $\emph{(ii)}$ shows that 
$$Tr\{e^{itH} - e^{itH_0} - D(e^{itH_0})\bullet V\} = it \int_0^1 ds Tr\{V\left(e^{itH_s} - e^{itH_o}\right)\} = (it)^2 \int_a^b e^{it\lambda}\eta(\lambda)d\lambda.$$
~~~~~~~~~~~~~~~~~~~~~~~~~~~~~~~~~~~~~~~~~~~~~~~~~~~~~~~~~~~~~~~~~~~~~~~~~~~~~~~~~~~~~~~~~~~~~~~~~~~~~~~~~~~~~~~~~~~~~~~~~~~~~~~~~~~~~~~~\end{proof}

\section{Reduction to finite dimension}\label{sec: third}

We begin with a proposition collecting some results, following from the  Weyl-von Neumann type construction.
\begin{ppsn}\label{prop: weylabs}
Let $A$ be a self-adjoint operator (possibly unbounded) in a separable infinite dimensional Hilbert space $\mathcal{H}$ and let $\{f_l\}_{1\leq l\leq L}$ be a set of normalized vectors in $\mathcal{H}$
and $\epsilon>0$.
\vspace{0.1in}

(i) Then there exists a finite rank projection $P$ such that $\|(I-P)f_l\|< \epsilon$ for $1\leq l\leq L$.
\vspace{0.1in}

(ii) Furthermore, $(I-P)AP \in \mathcal{B}_2(\mathcal{H}) ,~ \|(I-P)AP\|_2 < \epsilon$  and  $\|(I-P)e^{itA}P\|_2<\epsilon$  uniformly for $t$ with $|t|\leq T$.
\end{ppsn}
\begin{proof}
 Let $F(.)$ be the spectral measure associated with the self-adjoint operator $A$, and choose $a_l > 0$ such that 
\begin{equation*}
 \|\left[I-F\left((-a_l,a_l]\right)\right]f_l\|<\epsilon \quad \text{for} \quad 1\leq l\leq L.
\end{equation*}
If we set $a=\max\{a_l:~~~1\leq l\leq L\}$, then 
\begin{equation*}
 \|\left[I-F\left((-a,a]\right)\right]f_l\|\leq \|\left[I-F\left((-a_l,a_l]\right)\right]f_l\|< \epsilon \quad \text{for} \quad 1\leq l\leq L. 
\end{equation*}
For each positive integer $n$ and $1\leq k\leq n$, ~set $F_k=F\left(\left(\frac{2k-2-n}{n}a,\frac{2k-n}{n}a\right]\right)$ ~so that
\begin{equation*}
 F_kF_j=\delta_{kj}F_j \quad \text{and} \quad \sum_{k=1}^n F_k = F\left((-a,a]\right).
\end{equation*}
We also set for $1\leq k\leq n$ and $1\leq l\leq L$,
\begin{equation*}\label{eq: I est}
g_{kl} =
\begin{cases}
\frac{F_kf_l}{\|F_kf_l\|}, &\text{if $F_kf_l\neq 0.$}\\
 0, ~&\text{if $ F_kf_l=0$.}
\end{cases}
\end{equation*}
Let $P$ be the orthogonal projection onto the subspace generated by $\{g_{kl}: 1\leq k\leq n ; 1\leq l\leq L\}$ ; dim$P\mathcal{H} \leq nL$. 
Clearly $g_{kl}\in Dom(A)$ for all $k,l$ and hence $P\mathcal{H} \subseteq Dom(A)$. Moreover, $Ag_{kl} , PAg_{kl} \in F_k\mathcal{H}$ for each $k$ and $l$.
A simple calculation as in page 831 of (\cite{Sinhamohapatra}), shows that for $\lambda_k = \frac{2k-n-1}{n}a$,
\begin{equation*}
\|\left(A-\lambda_k\right)g_{kl}\|^2 \leq \left(\frac{a}{n}\right)^2 \quad \text{for} \quad 1\leq l\leq L, \quad \text{and therefore} \quad
\end{equation*}
\begin{equation*}
\|(I-P)APu\|^2 \leq ~\left(\frac{a}{n}\right)^2\sum_k\left(\sum_l|\langle u,g_{k,l}\rangle|\right)^2 \leq ~\frac{a^2}{n^2}~L~\|u\|^2 \quad \text{for} \quad u\in \mathcal{H}.
\end{equation*}
The operators $PA(I-P)$ and $(I-P)AP$ are finite rank operators with rank less than or equal to $nL$. Hence, using the above estimate we get that 
\begin{equation*}
 \|(I-P)AP\|_2 = \|PA(I-P)\|_2\leq \sqrt{dim(P)}~\|(I-P)AP\|\leq \sqrt{nL}~\left(\frac{a}{n}\right)~\sqrt{L}= L~\left(\frac{a}{\sqrt{n}}\right).
\end{equation*}
Thus again by the same calculation as in page 831 of (\cite{Sinhamohapatra}), it follows that 
\begin{equation}\label{eq: gronineq}
\alpha(t) \equiv \|(I-P)e^{itA}P\|_2 = \|(I-P)\left(e^{itA}-I\right)P\|_2 \leq a\sqrt{L}\int_0^t\alpha(s)ds + T~L~\frac{a}{\sqrt{n}} \quad \text{for} \quad |t|\leq T,
\end{equation}
solving this Gronwall-type inequality \eqref{eq: gronineq} leads to
\begin{equation*}
\alpha(t) ~\leq \frac{\left(T~L~a~e^{a\sqrt{L}t}\right)}{\sqrt{n}}~\leq \frac{\left(T~L~a~e^{a\sqrt{L}T}\right)}{\sqrt{n}}.
\end{equation*}
Since $(I-P)F\left((-a,a]\right)f_l = 0$ for $1\leq l\leq L$,
\begin{equation*}
\|(I-P)f_l\|  = \|(I-P)\left[I-F\left((-a,a]\right)\right]f_l\| ~\leq \|\left[I-F\left((-a,a]\right)\right]f_l\| < \epsilon \quad \text{for} \quad 1\leq l\leq L. 
\end{equation*}
The proof concludes by choosing $n$ sufficiently large.
~~~~~~~~~~~~~~~~~~~~~~~~~~~~~~~~~~~~~~~~~~~~~~~~\end{proof}

\begin{lmma}\label{lmma: finiteapp}
Let $H$ and $H_0$ be two self-adjoint operators in a separable infinite dimensional Hilbert space $\mathcal{H}$ such that  $H-H_0 \equiv V \in \mathcal{B}_2(\mathcal{H})$. Then
given $\epsilon>0$, there exists a projection P of finite rank such that for all $t$ with $|t|\leq T$,
\vspace{0.1in}

(i) $\|(I-P)H_0P\|_2 < \epsilon$, \hspace{1cm} $ \|(I-P)e^{itH_0}P\|_2 < \epsilon$,
\vspace{0.1in}

(ii) $\|(I-P)V\|_2 < 2\epsilon$, \hspace{1cm} $\|(I-P)HP\|_2 < 3\epsilon.$
\end{lmma}
\begin{proof}
Let $V=\suml_{l=1}^{\infty}\tau_l \lvert f_l \rangle \langle f_l \rvert$ be the canonical form of $V$ with $\suml_{l=1}^{\infty}\tau_l^2< \infty$ 
and choose $L$ in $V_L \equiv \suml_{l=1}^{L} \tau_l\lvert f_l \rangle \langle f_l \rvert$ so that $\|V-V_L\|_2 = \sqrt{\suml_{l=L+1}^{\infty}\tau_l^2}< \epsilon$ and 
$\epsilon ^{'} =\min \{\epsilon, \frac{\epsilon}{\suml_{l=1}^L|\tau_l|} \}> 0$. Next, we apply Proposition \ref{prop: weylabs} with  $A=H_0,~\{f_1,f_2,\ldots, f_L\}$ 
and $\epsilon^{'}$ in place of $\epsilon$. Hence we get a projection $P$ of finite rank in $\mathcal{H}$ such that 
\begin{equation*}
 \|(I-P)H_0P\|_2< \epsilon^{'}< \epsilon \quad \text{and} \quad \|(I-P)e^{itH_0}P\|_2< \epsilon^{'}< \epsilon, 
\end{equation*}
 uniformly for $t$ with $|t|\leq T$. For $\emph{(ii)}$ we note that  
\begin{equation*}
\|(I-P)V\|_2 \leq \|V-V_L\|_2 + \|(I-P)V_L\|_2  < \epsilon + \epsilon^{'}\left(\sum_{l=1}^L |\tau_l|\right) < 2\epsilon \quad \text{and therefore} \quad
\end{equation*}
\begin{equation*}
 \|(I-P)HP\|_2 \leq \|(I-P)H_0P\|_2 + \|(I-P)VP\|_2 < 3\epsilon.
\end{equation*}
~~~~~~~~~~~~~~~~~~~~~~~~~~~~~~~~~~~~~~~~~~~~~~~~~~~~~~~~~~~~~~~~~~~~~~~~~~~~~~~~~~~~~~~~~~~~~~~~~~~~~~~~~~~~~~~~~~~~~~~~~~~~~~~~~~~~~~~~~~~\end{proof}
\begin{rmrk}\label{rmk: sequence}
We can reformulate the statement of  Lemma \ref{lmma: finiteapp}  by saying that there exists a sequence $\{P_n\}$ of finite rank projections 
in $\mathcal{H}$ such that 
$$\|(I-P_n)H_0P_n\|_2,~ \|(I-P_n)e^{itH_0}P_n\|_2,~ \|(I-P_n)V\|_2 ,~ \|(I-P_n)HP_n\|_2 \longrightarrow 0 \quad \text{as} \quad n\longrightarrow \infty.$$
It may also be noted that $\{P_n\}$ does not necessarily converge strongly to $I$.
\end{rmrk}
\vspace{0.2in}

The next two theorems show how  Lemma \ref{lmma: finiteapp} can be used to reduce the relevant  problem into a finite dimensional one, in the cases when the self-adjoint pair $(H_0 , H)$ are bounded and unbounded.
\begin{thm}\label{thm: bdpolyappro}
Let $H$ and $H_0$ be two bounded self-adjoint operators in a separable infinite dimensional Hilbert space $\mathcal{H}$ such that $H-H_0 \equiv V \in \mathcal{B}_2(\mathcal{H})$. Then 
there exists a sequence $\{P_n\}$ of finite rank projections in $\mathcal{H}$ such that
\vspace{0.1in}

$Tr\{p(H)-p(H_0)-Dp(H_0)\bullet V\}$
\begin{equation}\label{eq: appropoly}
=\lim_{n\rightarrow \infty} Tr\{P_n\left[p(P_nHP_n)-p(P_nH_0P_n)-Dp(P_nH_0P_n)\bullet P_nVP_n\right]P_n\},
\end{equation}
where $p(.)$ is a polynomial.
\end{thm}
\begin{proof}
It will be sufficient to prove the theorem for $p(\lambda)=\lambda^r$. Note that for $r=0$ or $1$, both sides of \eqref{eq: appropoly} are identically 
zero. Using the sequence $\{P_n\}$ of finite rank projections as obtained in Lemma \ref{lmma: finiteapp} and using an expression similar to \eqref{eq: totalexpress}
in $\mathcal{B}(\mathcal{H})$, we have that
\begin{equation*}
\hspace{-10cm} Tr\{\left[p(H)-p(H_0)-Dp(H_0)\bullet V\right] 
\end{equation*}
\hspace{4cm} $- P_n\left[p(P_nHP_n)-p(P_nH_0P_n)-Dp(P_nH_0P_n)\bullet P_nVP_n\right]P_n\}$
\begin{equation*}
\hspace{0.2cm} = Tr\{\left[H^r-H_0^r-D(H_0^r)\bullet V\right] - P_n\left[(P_nHP_n)^r-(P_nH_0P_n)^r-D((P_nH_0P_n)^r)\bullet P_nVP_n\right]P_n\}
\end{equation*}
\begin{equation*}
\hspace{-9.9cm} = \sum_{j=0}^{r-2}\sum_{k=0}^{r-j-2}Tr\{ H^{r-j-k-2}VH_0^{k}VH_0^j 
\end{equation*}
\begin{equation*}
\hspace{4cm} - P_n(P_nHP_n)^{r-j-k-2}(P_nVP_n)(P_nH_0P_n)^k(P_nVP_n)(P_nH_0P_n)^jP_n\}
\end{equation*}
\begin{equation*}
\hspace{-4.7cm} = \sum_{j=0}^{r-2}\sum_{k=0}^{r-j-2} Tr\{ \left[H^{r-j-k-2} P_n - (P_nHP_n)^{r-j-k-2}\right]P_nVH_0^kVH_0^j
\end{equation*}
\begin{equation*}
\hspace{2cm} +~ H^{r-j-k-2} P^{\bot}_nVH_0^kVH_0^j + (P_nHP_n)^{r-j-k-2}P_nVP_n^{\bot}H_0^kVH_0^j
\end{equation*}

\begin{equation*}
\hspace{1cm} + ~(P_nHP_n)^{r-j-k-2}(P_nVP_n)\left[P_nH_0^k-(P_nH_0P_n)^k\right]VH_0^j
\end{equation*}

\begin{equation*}
\hspace{0.1cm} +~ (P_nHP_n)^{r-j-k-2}(P_nVP_n)(P_nH_0P_n)^kP_nVP_n^{\bot}H_0^j
\end{equation*}

\begin{equation}\label{eq: polysumexp}
\hspace{2cm} +~ (P_nHP_n)^{r-j-k-2}(P_nVP_n)(P_nH_0P_n)^k(P_nVP_n)\left[P_nH_0^j-(P_nH_0P_n)^j\right]\}.
\end{equation}
Using the results of Lemma \ref{lmma: finiteapp}, the first term of the expression \eqref{eq: polysumexp} leads to 
\begin{equation*}
\begin{split}
& \left\|\left[H^{r-j-k-2} - (P_nHP_n)^{r-j-k-2}\right]P_n\right\|_2 = \left\|\sum_{l=0}^{r-j-k-3} H^{r-j-k-l-3}(P_n^{\bot}HP_n)(P_nHP_n)^l\right\|_2\\
& \hspace{1cm} \leq (r-j-k-2)\|H\|^{r-j-k-3} \left\|P_n^{\bot}HP_n\right\|_2 \leq r (1+\|H\|)^r \left\|P_n^{\bot}HP_n\right\|_2 ,
\end{split}
\end{equation*}
which converges to 0 as $n\longrightarrow \infty$. For the fourth term in \eqref{eq: polysumexp}, we note that as in the calculations above, 
$$ \left\|P_n\left[H_0^k - (P_nH_0P_n)^k\right]\right\|_2 \leq k (1+\|H_0\|)^k \left\|P_n^{\bot}H_0P_n\right\|_2 \longrightarrow 0 \quad \text{as} \quad n\longrightarrow \infty, $$ and the sixth term is very similar to
the fourth term. The second, third and fifth terms in \eqref{eq: polysumexp} converges to zero in trace-norm since by Lemma \ref{lmma: finiteapp}, 
$\left\|P_n^{\bot}V\right\|_2 \longrightarrow 0 \quad \text{as} \quad n\longrightarrow \infty$.

~~~~~~~~~~~~~~~~~~~~~~~~~~~~~~~~~~~~~~~~~~~~~~~~~~~~~~~~~~~~~~~~~~~~~~~~~~~~~~~~~~~~~~~~~~~~~~~~~~~~~~~~~~~~~~~~~~~~~~~~~~~~~~~~~~~~\end{proof}

\begin{thm}\label{thm: expapps}
Let $H$ and $H_0$ be two self-adjoint operators (not necessarily bounded) in a separable infinite dimensional Hilbert space $\mathcal{H}$ such that $H-H_0 \equiv V \in \mathcal{B}_2(\mathcal{H})$. Then 
there exists a sequence $\{P_n\}$ of finite rank projections in $\mathcal{H}$ such that for any $T>0$
\vspace{0.1in}

$Tr\{e^{itH} - e^{itH_0} - D(e^{itH_0})\bullet V\}$
\begin{equation*}\label{eq: approexpo}
=\lim_{n\rightarrow \infty} Tr\{P_n\left[e^{itP_nHP_n} - e^{itP_nH_0P_n} - D(e^{itP_nH_0P_n})\bullet P_nVP_n\right]P_n\},
\end{equation*}
uniformly for all t with $|t|\leq T$.
\end{thm}
\begin{proof}
As in the case of a finite dimensional Hilbert space, $f(H_0) = e^{itH_0}$ is Frechet differentiable and 
$$D(e^{itH_0})\bullet V = it\int_0^1 e^{it\alpha  H_0} V e^{it(1-\alpha)  H_0} d\alpha \in \mathcal{B}_2(\mathcal{H}).$$ Therefore
\begin{equation}\label{eq: expfrechex}
\begin{split}
& e^{itH} - e^{itH_0} - D(e^{itH_0})\bullet V \\
& \hspace{2cm} = (it)^2 \int_0^1 \alpha d\alpha \int_0^1 d\beta ~e^{it\alpha \beta H}Ve^{it\alpha (1-\beta)H_0}Ve^{it(1-\alpha) H_0} \in \mathcal{B}_1(\mathcal{H})\\
\end{split}
\end{equation}
and hence by Fubini's theorem,
\begin{equation*}
\begin{split}
& Tr\{e^{itH} - e^{itH_0} - D(e^{itH_0})\bullet V\} \\
& \hspace{2cm} = (it)^2 \int_0^1 \alpha d\alpha \int_0^1 d\beta ~Tr\{e^{it\alpha \beta H}Ve^{it\alpha (1-\beta)H_0}Ve^{it(1-\alpha) H_0}\}. \\
\end{split}
\end{equation*}
Thus,
\begin{equation*}
\begin{split}
& Tr\{e^{itH} - e^{itH_0} - D(e^{itH_0})\bullet V\}\\
& \hspace{2cm} -  Tr\{P_n\left[e^{itP_nHP_n} - e^{itP_nH_0P_n} - D(e^{itP_nH_0P_n})\bullet P_nVP_n\right]P_n\}\\
& = (it)^2 \int_0^1 \alpha d\alpha \int_0^1 d\beta ~Tr\{e^{it\alpha \beta H}Ve^{it\alpha (1-\beta)H_0}Ve^{it(1-\alpha) H_0}\\
& \hspace{5cm} - P_ne^{it\alpha \beta P_nHP_n}P_nVP_ne^{it\alpha (1-\beta)P_nH_0P_n}P_nVP_ne^{it(1-\alpha) P_nH_0P_n}P_n\}\\
\end{split}
\end{equation*}
\begin{equation*}
\hspace{-2.5cm} = (it)^2 \int_0^1 \alpha d\alpha \int_0^1 d\beta ~Tr\{\left[e^{it\alpha \beta H} - e^{it\alpha \beta P_nHP_n}\right]P_nV e^{it\alpha (1-\beta)H_0}Ve^{it(1-\alpha) H_0}\\
\end{equation*}
\begin{equation*}
\hspace{3cm} +~~e^{it\alpha \beta H} P_n^{\bot}V  e^{it\alpha (1-\beta)H_0} V e^{it(1-\alpha) H_0}
\end{equation*}

\begin{equation*}
\hspace{3cm} +~P_ne^{it\alpha \beta P_nHP_n} P_nVP_n^{\bot}  e^{it\alpha (1-\beta)H_0} V e^{it(1-\alpha) H_0}
\end{equation*}

\begin{equation*}
\hspace{3cm} +~~P_ne^{it\alpha \beta P_nHP_n} P_nVP_n \left[e^{it\alpha (1-\beta)H_0} - e^{it\alpha (1-\beta)P_nH_0P_n}\right] P_nV e^{it(1-\alpha) H_0}
\end{equation*}

\begin{equation*}
\hspace{5cm} +~~P_ne^{it\alpha \beta P_nHP_n} P_nVP_n e^{it\alpha (1-\beta)H_0} P_n^{\bot}V e^{it(1-\alpha) H_0}
\end{equation*}

\begin{equation*}
\hspace{5cm} +~~P_ne^{it\alpha \beta P_nHP_n} P_nVP_n e^{it\alpha (1-\beta)P_nH_0P_n} P_nVP_n^{\bot} e^{it(1-\alpha) H_0}
\end{equation*}

\begin{equation}\label{eq: expexpress}
\hspace{3cm} +~~P_ne^{it\alpha \beta P_nHP_n} P_nVP_n e^{it\alpha (1-\beta)P_nH_0P_n} P_nVP_n \left[e^{it(1-\alpha)H_0} - e^{it(1-\alpha)P_nH_0P_n}\right]\}.
\end{equation}
In the first term of the expression \eqref{eq: expexpress} : 
\begin{equation*}
\begin{split}
& \left\| \left[e^{it\alpha \beta H} - e^{it\alpha \beta P_nHP_n}\right]P_n\right\|_2 \\
& \hspace{2cm} \leq \left\| it\alpha \beta \int_0^1 d\gamma e^{it\alpha \beta \gamma H} P_n^{\bot}HP_n e^{it\alpha \beta (1-\gamma)P_nHP_n}P_n\right\|_2 \leq ~T\left\|P_n^{\bot}HP_n\right\|_2,\\
\end{split}
\end{equation*}
which converges to $0$ as $n\rightarrow \infty$, uniformly for $|t|\leq T$ by Remark \ref{rmk: sequence} . For the fourth term in \eqref{eq: expexpress}, we note that as in the calculations above,
$$\left\| \left[e^{it\alpha (1-\beta) H_0} - e^{it\alpha (1-\beta) P_nH_0P_n}\right]P_n\right\|_2 \leq ~T \left\|P_n^{\bot}H_0P_n\right\|_2\longrightarrow 0 \quad \text{as} \quad n\longrightarrow \infty,$$ for $|t| \leq T$
and the seventh term is very similar to the fourth term. The second, third, fifth and sixth terms in \eqref{eq: expexpress} converges to zero in trace-norm
since by Lemma \ref{lmma: finiteapp}, 
$\left\|P_n^{\bot}V\right\|_2 \longrightarrow 0 \quad \text{as} \quad n\longrightarrow \infty$.
~~~~~~~~~~~~~~~~~~~~~~~~~~~~~~~~~~~~~~~~~~~~~~~~~~~~~~~~~~~~~~~~~~~~~~~~~~~~~~~~~~~~~~~~~~~~~~~\end{proof}

\section{Koplienko formula for both bounded and unbounded cases}
In this section, we derive the trace formulas for both bounded and unbounded self-adjoint pairs $(H_0, H)$.
\begin{thm}\label{thm: finalbddthm}
Let $H$ and $H_0$ be two bounded self-adjoint operators in an infinite dimensional separable Hilbert space $\mathcal{H}$ such that 
$H-H_0 \equiv V \in \mathcal{B}_2(\mathcal{H})$.  Then for any polynomial $p(.)$, $p(H) - p(H_0) - Dp(H_0)\bullet V \in \mathcal{B}_1(\mathcal{H})$ and
there exists a unique non-negative $L^1(\mathbb{R})$-function $\eta$ supported on $[a,b]$ such that 
$$Tr\{p(H)-p(H_0)-Dp(H_0)\bullet V\} = \int_a^b p''(\lambda)\eta (\lambda)d\lambda,$$ where,~$a= \inf \sigma(H_0) - \|V\|$,~
$b= \sup \sigma(H_0) + \|V\|.$ Furthermore $\itt_a^b|\eta(\lambda)|d\lambda = \frac{1}{2}\|V\|_2^2.$
\end{thm}
\begin{proof}
 By Theorem \ref{thm: bdpolyappro} and Theorem \ref{thm: finitekop}, we have that
\begin{equation*}
\begin{split}
& Tr\{p(H)-p(H_0)-Dp(H_0)\bullet V\}\\
&\hspace{1cm}  =\lim_{n\rightarrow \infty} Tr\{P_n\left[p(P_nHP_n)-p(P_nH_0P_n)-Dp(P_nH_0P_n)\bullet P_nVP_n\right]P_n\} \\
& \hspace{1cm} = \lim_{n\rightarrow \infty} \int_a^b p^{''}(\lambda) \eta_n(\lambda)d\lambda,\\ 
\end{split}
\end{equation*}

with $\eta_n(\lambda)$ given by \eqref{eq: itaexp}, and $\|\eta_n\|_1 = \frac{1}{2} \|P_n(H-H_0)P_n\|_2^2$, which clearly converges to 
\vspace{0.05in}

$\frac{1}{2}\|V\|_2^2$ as $n\rightarrow \infty$. Set $V_n \equiv P_nVP_n$; ~$H_n \equiv P_nHP_n$; ~$H_{0,n} \equiv P_nH_0P_n$ and $E_{0,n}(.)$, $E_{s,n}(.)$
\vspace{0.05in}

are the spectral families of $H_{0,n}$ and $H_{s,n} \equiv P_nH_sP_n$ respectively. Following the idea 
\vspace{0.05in}

contained in the paper of Gestezy et.al (\cite{gesztesykop}), using the expression \eqref{eq: itaexp} of $\eta_n$ and using 
\vspace{0.05in}

Fubini's theorem to interchange the orders of integration and integrating by-parts, we 
\vspace{0.05in}

have for $f\in L^{\infty}([a,b])$ and $g(\lambda) = \itt_a^{\lambda}f(\mu)d\mu$ that
\begin{equation*}
\begin{split}
& \int_a^b f(\lambda)\left[\eta_n(\lambda) - \eta_m(\lambda)\right]d\lambda \\
& \hspace{1cm} = \int_0^1 ds~ \int_a^b g'(\lambda)~ Tr\{V_n\left[E_{0,n}(\lambda) - E_{s,n}(\lambda)\right]\\
& \hspace{6cm} - V_m\left[E_{0,m}(\lambda) - E_{s,m}(\lambda)\right]\}d\lambda\\
& \hspace{1cm} = \int_0^1 ds ~\{g(\lambda) ~Tr(V_n\left[E_{0,n}(\lambda) - E_{s,n}(\lambda)\right]\\
& \hspace{6cm} - V_m\left[E_{0,m}(\lambda) - E_{s,m}(\lambda)\right])\}|_{a}^{b}\\
& \hspace{3cm} -\int_0^1ds\int_a^b ~g(\lambda) ~Tr\{V_n\left[E_{0,n}(d\lambda) - E_{s,n}(d\lambda)\right]\\
& \hspace{8cm} - V_m\left[E_{0,m}(d\lambda) - E_{s,m}(d\lambda)\right]\}\\
\end{split}
\end{equation*}
\be\label{eq: inteq}
\hspace{-0.1cm} = \int_0^1 ds ~Tr\{V_n\left[g(H_{s,n}) - g(H_{0,n})\right] - V_m\left[g(H_{s,m}) - g(H_{0,m})\right]\},
\ee
where we have noted that all the boundary terms vanishes. Next we note as in \eqref{eq: doubspec} that 
$$g(H_0) - g(H_s)  = -s \int_a^b \int_a^b \frac{g(\alpha) - g(\beta)}{\alpha - \beta} \mathcal{G}(d\alpha \times d\beta).V,$$
where $\mathcal{G}$ as earlier, defines a $\mathcal{B}_2(\mathcal{H})$-valued spectral measure in $\mathbb{R}^2$ with total $\mathcal{B}_2(\mathcal{H})$-variation less than or
equal to 1. Therefore ~$\|g(H_s) - g(H_0)\|_2 \leq s~\|f\|_{\infty} ~\|V\|_2$ ~~since 
$$\hspace{-7cm} \sup_{\alpha, \beta \in [a,b]; \alpha \neq \beta} \left|\frac{g(\alpha) - g(\beta)}{\alpha - \beta}\right|\leq \|f\|_{\infty}.\quad \text{Similarly} \quad$$  
$$\hspace{-2cm} \left\|P_n\left[g(H_{s,n}) - g(H_{s})\right]P_n\right\|_2 \leq \|f\|_{\infty}\left(\left\|P_n^{\bot}H_0P_n\right\|_2 + s\left\|P_n^{\bot}VP_n\right\|_2\right)\quad \text{and} \quad$$ 
$\hspace{.5cm} \left\|P_n\left[g(H_{0,n}) - g(H_{0})\right]P_n\right\|_2 \leq \|f\|_{\infty}\left\|P_n^{\bot}H_0P_n\right\|_2.$ Therefore
\begin{equation*}
\begin{split}
& \left|\int_a^b f(\lambda)\left[\eta_n(\lambda) - \eta_m(\lambda)\right]d\lambda \right| \\ 
&\hspace{1cm}  = |\int_0^1 ds ~(Tr( V_n\{\left[g(H_{s,n}) - g(H_{0,n})\right] - \left[g(H_{s}) - g(H_{0})\right]\})\\
& \hspace{3cm} - Tr( V_m\{\left[g(H_{s,m}) - g(H_{0,m})\right] - \left[g(H_{s}) - g(H_{0})\right]\})\\
& \hspace{5cm} + Tr\{(V_n - V_m)\left[g(H_s) - g(H_0)\right]\})|\\
& \hspace{1cm} \leq \|f\|_{\infty} \|V\|_2 ~(\int_0^1 ds \{ 2 \left(\left\|P_n^{\bot}H_0P_n\right\|_2 + \left\|P_m^{\bot}H_0P_m\right\|_2\right)\\
& \hspace{4cm} + s\left(\left\|P_n^{\bot}VP_n\right\|_2 + \left\|P_m^{\bot}VP_m\right\|_2\right) + s\|V_n - V_m\|_2\}).\\
\end{split}
\end{equation*}
So, by Hahn-Banach theorem, $\{\eta_n\}$ is a Cauchy sequence of non-negative functions in $L^1([a,b])$ and hence there exists a non-negative $L^1([a,b])$-
function $\eta$ such that $\{\eta_n\}$ converges to $\eta$ in $L^1$-norm. Thus
$$ Tr\{p(H)-p(H_0)-Dp(H_0)\bullet V\} = \lim_{n\rightarrow \infty} \int_a^b p^{''}(\lambda) \eta_n(\lambda)d\lambda = \int_a^b p^{''}(\lambda)\eta(\lambda)d\lambda.$$
The uniqueness of $\eta$ follows from the uniqueness of a probability density, supported on a finite interval in $\mathbb{R}$, with a given sequence of
moments (\cite{krpbook}).
~~~~~~~~~~~~~~~~~~~~~~~~~~~~~~~~~~~~~~~~~~~~~~~~~~~~~~~~~~~~~~~\end{proof}

\begin{lmma}\label{lmma: expkoplmma}
Let $H$ and $H_0$ be two self-adjoint operators in an infinite dimensional separable Hilbert space $\mathcal{H}$ such that 
$H-H_0 \equiv V\in \mathcal{B}_2(\mathcal{H})$. Then $ e^{itH} - e^{itH_0} - D(e^{itH_0})\bullet V \in \mathcal{B}_1(\mathcal{H})$ and there exists a unique non-negative $L^1(\mathbb{R})$-function $\eta$ such that
$$ Tr\{e^{itH} - e^{itH_0} - D(e^{itH_0})\bullet V\} = (it)^2 \int_{\mathbb{R}} e^{it\lambda} \eta(\lambda) d\lambda.$$
\end{lmma}
\begin{proof}
The uniqueness part is trivial, since if not let $\eta_1$ and $\eta_2$ be two such functions so that
$$\int_{\mathbb{R}} e^{it\lambda}\left[\eta_1(\lambda) - \eta_2(\lambda)\right] d\lambda = 0 ~~\forall ~~t\in \mathbb{R} \quad \text{and} \quad \eta_1-\eta_2\in L^1(\mathbb{R}).$$

Then by Fourier Inversion Theorem we conclude that $\eta_1 = \eta_2$ a.e. By Theorem \ref{thm: expapps} we 
\vspace{0.05in}

conclude that, there exists a sequence $\{P_n\}$ of finite rank projections such that
\begin{equation}\label{eq: finiteapprox}
Tr\{e^{itH} - e^{itH_0} - D(e^{itH_0})\bullet V\} = \lim_{n\rightarrow \infty} Tr\{P_n\left[e^{itH_n} - e^{itH_{0,n}} - D(e^{itH_{0,n}})\bullet V_n\right]P_n\},
\end{equation}

where $H_n \equiv P_nHP_n,~ H_{0,n} \equiv P_nH_0P_n$ and $V_n \equiv P_nVP_n$, and the convergence is uniform  
\vspace{0.1in}

in $t$ for $|t| \leq T$. Note that by construction $P_n\mathcal{H} \subseteq Dom(H_0) = Dom(H)$ (see proof of   
\vspace{0.05in}

Proposition \ref{prop: weylabs}) and hence both $H_n$ and $H_{0,n}$ are self-adjoint operators in the finite 
\vspace{0.05in}

dimensional space $P_n\mathcal{H}$. By \eqref{eq: finiteexp}, there exists a unique non-negative $\eta_n \in L^1(\mathbb{R})$ such 

that 
\begin{equation}\label{eq: finitekop}
 Tr\{P_n\left[e^{itH_n} - e^{itH_{0,n}} - D(e^{itH_{0,n}})\bullet V_n\right]P_n\} = (it)^2\int_{-\infty}^{\infty} e^{it\lambda} \eta_n(\lambda)d\lambda,
\end{equation}

and hence
\begin{equation}\label{eq: finiteexpress}
 Tr\{e^{itH} - e^{itH_0} - D(e^{itH_0})\bullet V\} = (it)^2 \lim_{n\rightarrow \infty}\int_{-\infty}^{\infty} e^{it\lambda} \eta_n(\lambda)d\lambda,
\end{equation}

the convergence being uniform in $t$ for $|t|\leq T$. In order to prove the $L^1(\mathbb{R})$-convergence 
\vspace{0.05in}

of $\{\eta_n\}$, we essentially repeat the procedure in the last part of Section \ref{sec: third} except that one
\vspace{0.05in}

needs to take into account the possibility that the indefinite integral $g$ of a $L^{\infty}(\mathbb{R})$- 
\vspace{0.05in}

function $f$ may have a linear part, which will make $g(H_0)$ and $g(H)$ unbounded operators. 
\vspace{0.01in}

Let $f = f_1 + if_2 \in L^{\infty}(\mathbb{R})$so that $f_j \in L^{\infty}(\mathbb{R})~~(j=1,2)$ with $\|f_j\|_{\infty} \leq \|f\|_{\infty}$ and if we 

set $$g(\lambda) = \int_0^{\lambda} f(\mu)d\mu + C = \int_0^{\lambda} \{f_1(\mu) + if_2(\mu)\}d\mu + (C_1+iC_2) = g_1(\lambda) + g_2(\lambda),$$

where $g_j(\lambda) = \itt_0^{\lambda}f_j(\mu)d\mu + C_j$ (for $j=1,2$) are real valued functions and ~$C_1, C_2$~ are 

some real constants. Then ~$-\|f\|_{\infty}|\lambda| + C_j ~\leq ~g_j(\lambda) ~\leq ~\|f\|_{\infty}|\lambda| + C_j$ (for $j=1,2$), and 
\vspace{0.01in}

by functional calculus we conclude that for any self-adjoint operator $A$, $Dom\left(g_j(A)\right) = $ 
\vspace{0.05in}

$Dom(A)$ and $g_j(A) - \|f\|_{\infty}A \in \mathcal{B}(\mathcal{H})$, for $j=1,2$. 

Thus
$$\hspace{-5cm} g_j(H_0+sV) - g_j(H_0) = \{ \left[g_j(H_0+sV) - \|f\|_{\infty} (H_0+sV)\right]$$ 
\hspace{8cm} $- \left[g_j(H_0) - \|f\|_{\infty} H_0\right] + \|f\|_{\infty} s V\} ~~\in \mathcal{B}(\mathcal{H}),$
\vspace{0.05in}

for $j=1,2$. By a similar calculation as in the proof of Theorem \ref{thm: finalbddthm}, it follows that 
$$g_j(H_0+sV) - g_j(H_0) \in \mathcal{B}_2(\mathcal{H})\quad \text{and} \quad \left\|g_j(H_0+sV) - g_j(H_0)\right\|_2 \leq s \|f\|_{\infty} \|V\|_2,$$

for $j=1,2$. Since $g=g_1 + ig_2$, we conclude that 
$$g(H_0+sV) - g(H_0) \in \mathcal{B}_2(\mathcal{H})\quad \text{and} \quad \left\|g(H_0+sV) - g(H_0)\right\|_2 \leq 2s \|f\|_{\infty} \|V\|_2.$$ 

Similarly, $$\left\|P_n\left[g(H_{s,n}) - g(H_{s})\right]P_n\right\|_2 \leq 2 \|f\|_{\infty}\left(\left\|P_n^{\bot}H_0P_n\right\|_2 + s\left\|P_n^{\bot}VP_n\right\|_2\right)\quad \text{and} \quad$$ 
$\left\|P_n\left[g(H_{0,n}) - g(H_{0})\right]P_n\right\|_2 \leq ~2\|f\|_{\infty}\left\|P_n^{\bot}H_0P_n\right\|_2.$ Also we get that 
$$\int_{\mathbb{R}} f(\lambda)\left[\eta_n(\lambda) - \eta_m(\lambda)\right]d\lambda = \int_0^1 ds ~Tr\{V_n\left[g(H_{s,n}) - g(H_{0,n})\right] - V_m\left[g(H_{s,m}) - g(H_{0,m})\right]\},$$
with the boundary term vanishing because for fixed finite $m$ and $n$, the support of the spectral measures involved are compact. Therefore
\begin{equation*}
\begin{split}
& \left|\int_a^b f(\lambda)\left[\eta_n(\lambda) - \eta_m(\lambda)\right]d\lambda \right| \\ 
&\hspace{1cm}  = |\int_0^1 ds ~(Tr( V_n\{\left[g(H_{s,n}) - g(H_{0,n})\right] - \left[g(H_{s}) - g(H_{0})\right]\})\\
& \hspace{3cm} - Tr( V_m\{\left[g(H_{s,m}) - g(H_{0,m})\right] - \left[g(H_{s}) - g(H_{0})\right]\})\\
& \hspace{5cm} + Tr\{(V_n - V_m)\left[g(H_s) - g(H_0)\right]\})|\\
& \hspace{1cm} \leq 2 \|f\|_{\infty} \|V\|_2 ~(\int_0^1 ds \{ 2 \left(\left\|P_n^{\bot}H_0P_n\right\|_2 + \left\|P_m^{\bot}H_0P_m\right\|_2\right)\\
& \hspace{4cm} + s\left(\left\|P_n^{\bot}VP_n\right\|_2 + \left\|P_m^{\bot}VP_m\right\|_2\right) + s\|V_n - V_m\|_2\}).\\
\end{split}
\end{equation*}
Therefore, by Remark \ref{rmk: sequence} and the Hahn-Banach theorem, $\{\eta_n\}$ is a Cauchy sequence of non-negative functions in $L^1(\mathbb{R})$ and hence there exists a non-negative $L^1(\mathbb{R})$-
function $\eta$ such that $\{\eta_n\}$ converges to $\eta$ in $L^1$-norm. Thus
$$Tr\{e^{itH} - e^{itH_0} - D(e^{itH_0})\bullet V\} = (it)^2 \lim_{n\rightarrow \infty}\int_{\mathbb{R}} e^{it\lambda} \eta_n(\lambda)d\lambda = (it)^2 \int_{\mathbb{R}}e^{it\lambda}\eta(\lambda)d\lambda.$$

~~~~~~~~~~~~~~~~~~~~~~~~~~~~~~~~~~~~~~~~~~~~~~~~~~~~~~~~~~~~~~~~~~~~~~~~~~~~~~~~~~~~~~~~~~~~~~~~~~~~~~~~~~~~~~~~~~~~~~~~~~~~~~~~~~~~~\end{proof}

\begin{thm}
Let $H$ and $H_0$ be two self-adjoint operators in an infinite dimensional separable Hilbert space $\mathcal{H}$ such that 
$H-H_0 \equiv V\in \mathcal{B}_2(\mathcal{H})$ and $f\in \mathcal{S}(\mathbb{R})$(the Schwartz class of smooth functions of rapid decrease). 
Then $f(H) - f(H_0) - Df(H_0)\bullet V \in \mathcal{B}_1(\mathcal{H})$ and
\begin{equation*}
Tr\{f(H) - f(H_0) - Df(H_0)\bullet V\} = \int_{\mathbb{R}}f^{''}(\lambda)\eta(\lambda)d\lambda,
\end{equation*}
where $\eta$ is a unique non-negative $L^1(\mathbb{R})$ -function with $\| \eta \|_1 = \frac{1}{2} \|V\|_2^2.$
\end{thm}
\begin{proof}
By the spectral theorem and an application of Fubini's theorem, we get that
\begin{equation*}
\hspace{-4cm} f(H) = \int_{\mathbb{R}} \hat{f}(t)e^{itH} dt~ ,~f(H_0) = \int_{\mathbb{R}} \hat{f}(t)e^{itH_0} dt \quad \text{and} \quad 
\end{equation*}
$$\hspace{2cm} Df(H_0)\bullet V = \int_{\mathbb{R}} \hat{f}(t)\left[D\left(e^{itH_0}\right)\bullet V\right]dt.$$
Thus using the expression \eqref{eq: expfrechex}, and the fact that $\hat{f}\in \mathcal{S}(\mathbb{R})$, and Fubini's theorem we conclude that 
$f(H) - f(H_0) - Df(H_0)\bullet V \in \mathcal{B}_1(\mathcal{H})$ and
\begin{equation*}
\begin{split}
& Tr\{f(H) - f(H_0) - Df(H_0)\bullet V\} = \int_{\mathbb{R}}\hat{f}(t) ~Tr\{e^{itH} - e^{itH_0} - D(e^{itH_0})\bullet V\}~dt \\
&\hspace{6.1cm}  = \int_{\mathbb{R}} \hat{f}(t)\left((it)^2\int_{\mathbb{R}}e^{it\lambda}\eta(\lambda)d\lambda\right)dt = \int_{\mathbb{R}} f^{''}(\lambda)\eta(\lambda)d\lambda,\\
\end{split}
\end{equation*}
where $\eta$ is the one obtained in Lemma \ref{lmma: expkoplmma}.
~~~~~~~~~~~~~~~~~~~~~~~~~~~~~~~~~~~~~~~~~~~~~~~~~~~~~~~~~~~~~~~~~~~~~~~~~~~~~\end{proof}

\begin{rmrk}
If $f(\lambda) = \itt_{-\infty}^{\infty}\frac{e^{it\lambda} - 1 - it\lambda}{(it)^2}\nu(dt) + \widetilde{C_1}\lambda + \widetilde{C_2}$,

\hspace{-0.8cm} where $\widetilde{C_1},\widetilde{C_2}$ are some constants and 
$\nu$ is a complex measure, then $f(H)$ and $f(H_0)$ are unbounded operators where $Domf(H), Domf(H_0)$ are contained in $Dom(H^2),Dom(H_0^2)$ respectively.
It may so happen that $Dom(H^2) \bigcap Dom(H_0^2)$ is not dense and in that case $f(H) - f(H_0)$ is not well-defined. But on the other hand,
\begin{equation*}
\begin{split}
& f(H) - f(H_0) - Df(H_0)\bullet V  = \int_{-\infty}^{\infty} \frac{1}{(it)^2}\left[e^{itH} - e^{itH_0} - D(e^{itH_0})\bullet V\right]\nu(dt) \\
&\hspace{5.2cm}  = \int_{-\infty}^{\infty}\nu(dt) \int_0^1 \alpha d\alpha \int_0^1 d\beta ~e^{it\alpha \beta H}Ve^{it\alpha (1-\beta)H_0}Ve^{it(1-\alpha) H_0},\\
\end{split}
\end{equation*}
in which the right hand side is a well-defined $\mathcal{B}_1(\mathcal{H})$ operator and one will have the result in this extended sense (i.e. whenever the expression in $\{.\}$
is densely defined):
$$Tr\{f(H) - f(H_0) - Df(H_0)\bullet V\} = \int_{-\infty}^{\infty}f^{''}(\lambda)\eta(\lambda)d\lambda.$$
\end{rmrk}

\noindent\textbf{Acknowledgment:} The authors would like to thank Council of Scientific and Industrial Research (CSIR), Government of India 
for a research and Bhatnagar Fellowship respectively. The authors also thank UK-India Education and Research Initiative ( UKIERI) project for support.

\end{document}